\documentclass[smallextended]{svjour3}   

\titlerunning{Multiple Solutions of Nonlinear Differential Equations}

\newcommand{\Reals}[1]{{\rm I\! R}^{#1}}
\usepackage[utf8]{inputenc}
\usepackage{amsmath}
\usepackage{graphicx} 
\usepackage{mathtools}
\usepackage{siunitx}
\usepackage{wrapfig}
\usepackage{algpseudocode}
\usepackage{algorithm,bm}
\usepackage{subeqnarray} 
\usepackage{hyperref}
\hypersetup{
    colorlinks=true,
    linkcolor=blue,
    filecolor=magenta,      
    urlcolor=cyan,
    pdftitle={Overleaf Example},
    pdfpagemode=FullScreen,
    }
\usepackage[utf8]{inputenc}
\usepackage{ amssymb, amscd, latexsym}
\usepackage{changepage}
\usepackage{marvosym}
\graphicspath{{./Figs/}}
\usepackage{color}
\usepackage[right]{lineno}
\usepackage{microtype}
\usepackage{tabu,tabularx,multirow,booktabs}
\usepackage{cuted}
\usepackage{stackengine}
\usepackage{subeqnarray} 
\usepackage{soul}
\stackMath
\newcommand\tenq[2][1]{%
 \def\useanchorwidth{T}%
  \ifnum#1>1%
    \stackunder[0pt]{\tenq[\numexpr#1-1\relax]{#2}}{\scriptscriptstyle\sim}%
  \else%
    \stackunder[1pt]{#2}{\scriptscriptstyle\sim}%
  \fi%
}
\usepackage{titlesec}
\usepackage{lscape}
\usepackage{indentfirst}
\usepackage{latexsym}
\usepackage{verbatim}
\usepackage{setspace}

\usepackage{graphicx,epstopdf} 
\usepackage[caption=false]{subfig} 

\usepackage{enumitem}
\usepackage{tikz}
\usetikzlibrary{calc}
\usepackage{xcolor}

\def\text#1{{\rm #1}}

\usepackage{cancel}

\newcommand{\vx}{{\bf{x}}}

\usepackage[rightcaption]{sidecap}

\newtheorem{thm}{Theorem}[section]

\newtheorem{assump}[thm]{Assumption}

\title{Companion-Based Multi-Level Finite Element Method for Computing Multiple Solutions of Nonlinear Differential Equations} 
\author{Wenrui Hao$^{1}$, Sun Lee$^{1}$, Young Ju Lee$^{2}$ }
\institute {1 Department Of Mathematics, Penn State, State College, PA 16802, USA  \and 2 Department Of Mathematics, Texas State, San Marcos, TX 78666, USA}
\begin{document}

\maketitle
\begin{abstract}
The use of nonlinear PDEs has led to significant advancements in various fields, such as physics, biology, ecology, and quantum mechanics. However, finding multiple solutions for nonlinear PDEs can be a challenging task, especially when suitable initial guesses are difficult to obtain. In this paper, we introduce a novel approach called the Companion-Based Multilevel finite element method (CBMFEM), which can efficiently and accurately generate multiple initial guesses for solving nonlinear elliptic semi-linear equations with polynomial nonlinear terms using finite element methods with conforming elements.
We provide a theoretical analysis of the error estimate of finite element methods using an appropriate notion of isolated solutions, for the nonlinear elliptic equation with multiple solutions and present numerical results obtained using CBMFEM which are consistent with the theoretical analysis.


\keywords{Elliptic Semilinear PDEs  \and Finite Element Method \and Multiple Solutions}
\subclass{49M37 \and 65N30 \and 90C99}

\end{abstract}
\section{Introduction}

Nonlinear partial differential equations (PDEs) are widely used in various fields, and there are many versions of PDEs available. One such example is Reaction-Diffusion Equations, which find applications in physics, population dynamics, ecology, and biology. In physics, Simple kinetics, Belousov–Zhabotinskii reactions, and Low-temperature wave models are examples of applications. In population dynamics and ecology, the Prey-predator model and Pollution of the environment are relevant. In biology, Reaction-diffusion equations are used to study Cell dynamics and Tumor growth \cite{volpert2014elliptic}. Schrodinger equations \cite{kevrekidis2015defocusing,wang2015new} and Hamiltonian systems \cite{kapitula2004counting,simon1995concentration} are also important topics in the field of Quantum mechanics. Another important area in the realm of nonlinear PDEs is pattern formation, which has numerous applications, such as the Schnakenberg model \cite{deutsch2005mathematical}, the Swift-Hohenberg equation \cite{lega1994swift}, the Gray-Scott model \cite{wei2013mathematical}, the FitzHugh-Nagumo equation \cite{jones2009differential}, and the Monge–Ampère \cite{figalli2017monge,gutierrez2001monge} equation, which finds various applications.

In this paper, we focus on computing multiple solutions of elliptic semi-linear equations with nonlinear terms expressed as polynomials. Although we limit our nonlinear term to a polynomial, it is still an interesting case of the above applications that has yet to be fully explored.

To solve these nonlinear PDEs, various numerical methods have been developed, such as Newton's method and its variants, Min Max method, bifurcation methods \cite{zhao2022bifurcation}, multi-grid method \cite{henson2003multigrid,xu1996two} or subspace correction method \cite{chen2020convergence} or a class of special two-grid methods \cite{cai2009numerical,huang2016newton,xu1994novel,xu1996two,xu2022new}, deflation method \cite{farrell2015deflation}, mountain pass method \cite{breuer2003multiple,choi1993mountain}, homotopy methods \cite{chen2008homotopy,hao2014bootstrapping,hao2020homotopy,hao2020spatial,wang2018two}, and Spectral methods \cite{grandclement2009spectral}. However, finding multiple solutions can be a challenging task, primarily due to the difficulty of obtaining suitable initial guesses for multiple solutions. It is often uncertain whether good numerical initial guesses for each solution exist that can converge to the solutions. Even if such initial guesses exist, finding them can be a challenging task. 

To address this challenge, we introduce a novel approach called the Companion-Based Multilevel finite element method (CBMFEM) for solving nonlinear PDEs using finite element methods with conforming elements. Our method is based on the structure of the full multigrid scheme \cite{brandt2011multigrid} designed for the general nonlinear elliptic system. Given a coarse level, we compute a solution using a structured companion matrix, which is then transferred to the fine level to serve as an initial condition for the fine level. We use the Newton method to obtain the fine-level solution for each of these initial guesses and repeat this process until we obtain a set of solutions that converge to the stationary solutions of the PDE. Our approach is different in literature, such as those presented in \cite{breuer2003multiple} or \cite{li2017new}, which attempt to find additional solutions based on the previously found solutions.  

The main advantage of our method is that it can generate multiple initial guesses efficiently and accurately, which is crucial for finding multiple solutions for nonlinear PDEs. Furthermore, our method is robust and can be easily applied to a wide range of elliptic semi-linear equations with polynomial nonlinear terms. 

In this paper, we also present a mathematical definition of isolated solutions for elliptic semilinear PDEs with multiple solutions, which leads to well-posedness of the discrete solution and provides a priori error estimates of the finite element solution using the framework introduced in \cite{xu1994novel} and \cite{xu1996two}.

We organize the paper as follows: In \S \ref{maingov}, we introduce the governing equations and basic assumptions. In \S \ref{fem}, we present the error estimate of the nonlinear elliptic equation using the FEM method. In \S \ref{CBMFEM}, we introduce the CBMFEM, including the construction of the companion matrix and filtering conditions. Finally, in \S \ref{num}, we present numerical results obtained using CBMFEM, which are consistent with the theoretical analysis.
Throughout the paper, we use standard notation for Sobolev spaces $W^{k,p}(\Omega)$ and the norm $\|\cdot\|_{k,p}$. If $k = 0$, then $\|\cdot\|_{0,p}$ denotes the $L^p$ norm. The symbol $C^k(\Omega)$ denotes the space of functions, whose first $k \geq 0$ derivatives are continuous on $\Omega$. Additionally, we denote $\tenq{v}$ as the vector while $\tenq[2]{v}$ is the matrix.

\section{Governing Equations}\label{maingov} 
In this section, we introduce the governing equations that we will be solving. Specifically, we are interested in solving the quasi-linear equations, where we assume that polygonal (polyhedral) domain $\Omega$ is a bounded domain in $\Reals{d}$ with $d = 1, 2,$ or $3$.

\begin{equation}\label{Eq1}
-\Delta u + f(x, u) = 0, \quad \mbox{ in } \Omega,
\end{equation}
subject to the following general mixed boundary condition:
\begin{equation}\label{main:bdry} 
\alpha u - \beta \nabla u \cdot \textbf{n} = \gamma g, \quad \mbox{ on } \partial \Omega,
\end{equation}
where $\textbf{n}$ is the unit outward normal vector to $\partial \Omega$, $\alpha$, $\beta$ and $\gamma$ are functions that can impose condition of $u$, on $\partial \Omega$, such as the Dirichlet boundary $\Gamma_D$, pure Neumann boundary $\Gamma_N$, mixed Dirichlet and Neumann boundary or the Robin boundary $\Gamma_R$, i.e.,  
\begin{equation}
\partial \Omega = \overline{\Gamma}_D \cup \overline{\Gamma}_N \quad \mbox{ or } \quad \partial \Omega = \overline{\Gamma}_R,
\end{equation}
with $\overline{\Gamma}_D$, $\overline{\Gamma}_N$, and $\overline{\Gamma}_R$ being the closure of $\Gamma_D$, $\Gamma_N$, and $\Gamma_R$, respectively. Specifically, 
\begin{equation}
\beta|_{\Gamma_D} = 0, \quad \alpha|_{\Gamma_N} = 0, \quad \mbox{ and } \quad (\alpha \beta)|_{\Gamma_R} > 0. 
\end{equation}
We shall assume that $g$ is smooth, and in particular when $\overline{\Gamma}_N = \partial \Omega$, compatibility conditions will be assumed for the functions $f$ and $g$ if necessary,  \cite{leykekhman2017maximum}. For the sake of convenience, we denote $\Gamma_{1} = \Gamma_R$ or $\Gamma_N$ throughout this paper and assume that for some $c > 0$, a generic constant, the followings hold:  
\begin{equation}
\beta = \gamma \quad \mbox{ and } \quad 0 \leq \frac{\alpha}{\beta} \leq c \quad \mbox{ on } \quad \Gamma_{1}.  
\end{equation}
We also provide some conditions for the function $f$, which is generally a nonlinear polynomial function in both $x$ and $u$. We shall assume that $f(x,u) : \overline{\Omega} \times \Reals{} \mapsto \Reals{}$ is smooth in the second variable. We shall denote the $k$-th derivative of $f$ with respect to $u$ by $f^{(k)}$, i.e. $f^{(k)} = \frac{\partial^k f}{\partial u^k}$.   and that there are positive constants $C_1$ and $C_2$ such that
\begin{equation}\label{bdf} 
|f(x,u)| \leq C_1 + C_2 |u|^{q},
\end{equation}
where $q$ is some real value, such that $1 \leq q \leq \infty$ for $d = 1$, $1 \leq q < \infty$ for $d = 2$ and $1 \leq q < 5$ for $d = 3$. This is sufficient for defining the weak formulation (see \eqref{main:eq}). To apply the finite element method, we consider the weak formulation of Eq. \eqref{Eq1} which satisfies the fully elliptic regularity (see \cite{schatz1996some} and references cited therein), i.e., solution is sufficiently smooth. 
We introduce a space $V$ defined by:\begin{equation}
V = \{ v \in H^1(\Omega) : v|_{\Gamma_D} = 0 \}. 
\end{equation}
The main problem can then be formulated as follows: Find $u \in V$ such that 
\begin{equation}\label{main:eq} 
\mathcal{F}(u,v) = a(u,v) + b(u,v) = 0, \quad \forall v \in V,  
\end{equation}
where $a(\cdot,\cdot), b(\cdot,\cdot) : V \times V \mapsto \Reals{}$ are the mappings defined as follow, respectively:  
\begin{subeqnarray}
a(u,v) &=& \int_\Omega \nabla u \cdot  \nabla v \, dx + \int_{\Gamma_{1}} \frac{\alpha}{\beta} u v \, ds, \quad \forall u, v \in V \\
b(u,v) &=& \int_\Omega f(x,u) v\, dx  - \int_{\Gamma_{1}} g v\, ds, \quad \forall u, v \in V. 
\end{subeqnarray}

\section{Finite element formulation and a priori error analysis}\label{fem} 

We will utilize a finite element method to solve \eqref{main:eq}, specifically a conforming finite element of degree $r \geq 1$. The triangulation of the domain $\Omega$ will be denoted by  $\mathcal{T}_h = \{T\}_{i=1,\cdots,N_e}$. As usual, we define
\begin{equation}
h = {\rm{max}}_{T \in \mathcal{T}_h}{\rm diam}(T).
\end{equation} 
Let $V_h$ be the subspace of $V$ that is composed of piecewise globally continuous polynomials of degree $r \geq 1$. We shall denote the set of vertices of $\mathcal{T}_h$ by $\mathcal{V}_h$. Then, the dimension of $V_h$ is denoted by $N_h$, the total number of interior vertices and the space $V_h$ can be expressed as follows:
\begin{equation} 
V_h = {\rm span} \{ \phi_h^i : i = 1,\cdots, N_h\},
\end{equation} 
where $\phi_h^i$ is the nodal basis on the triangulation $\mathcal{T}_h$, i.e., 
\begin{equation}
\phi_h^i(x_j) = \delta_{ij}, \quad \forall x_j \in \mathcal{V}_h.
\end{equation}
The discrete weak formulation for \eqref{main:eq} is given as: Find $u_h \in V_h$ such that
\begin{equation}\label{main:heq} 
\mathcal{F}(u_h,v_h) = a(u_h,v_h) + b(u_h,v_h) = 0, \quad \forall v_h \in V_h.  
\end{equation}
We note that for any $u_h \in V_h$, there exists a unique $\tenq{u}_h = (u_h^1, \cdots, u_h^{N_h})^T \in \Reals{N_h}$ such that 
\begin{equation}\label{expression}
u_h = \sum_{i=1}^{N_h} u_h^i \phi_h^i.   
\end{equation}
To obtain a solution $u_h$ to \eqref{main:heq}, we  need to solve the following system of nonlinear equations: 
\begin{equation}\label{goal1}
\tenq{F}_h(\tenq{u}_h) = \left ( 
\begin{array}{c}
F_h^1(\tenq{u}_h) \\ 
\vdots \\ 
F_h^{N_h}(\tenq{u}_h)
\end{array} \right ) 
= \tenq{0},  
\end{equation}
where
\begin{equation}\label{system}
F_h^i(\tenq{u}_h) := a(u_h, \phi_h^i) + b(u_h, \phi_h^i) = 0, \quad \forall i = 1,\cdots,N_h. 
\end{equation}

\subsection{A priori error analysis} 
In this section, we will discuss the convergence order of the finite element solutions for solving \eqref{main:eq}. Throughout this section, we introduce a notation for a fixed $\delta > 0$:
\begin{equation}
N_\delta = \{v \in V : \|u - v\|_{1,2} < \delta \},  
\end{equation}
where $u$ is the solution to the equation \eqref{main:eq}. We will make the following assumption:
\begin{assump}\label{assc} 
There exists a solution $u \in W^{1,2}(\Omega) \cap C^1(\Omega)$ of the problem \eqref{main:eq} and there is a constant $\Gamma$ such that $\|u\|_{0,\infty} \leq \Gamma$ and sufficiently smooth. Furthermore, in particular, $u$ is isolated in the following sense: there exists $\delta = \delta_u > 0$ such that for all $w \in V$ such that $0 \neq \|w\|_{1,2} < \delta$, there exist $\eta_u > 0$ and $v_w \in V$, such that 
\begin{equation}\label{iso1} 
\left |\mathcal{F}(u + w,v_w) \right | \geq \eta_u \|w\|_{1,2}\|v_w\|_{1,2} > 0.   
\end{equation} 
\end{assump}
\begin{remark}
We note that, to the best of our knowledge, this is the first time the notion of isolation has been introduced in the literature. In \cite{chen2008analysis}, a similar definition is presented, but it allows $v$ to be any function in $V$, not necessarily dependent on $w$ in \eqref{iso1}. This can lead to several issues. To illustrate one of the issues, we consider the problem of solving the following equation:
\begin{equation}\label{ex} 
-u'' - u^p = 0 \quad \mbox{ in } \quad \left ( 0, \frac{2\pi}{\sqrt{p}} \right ) \quad \mbox{ and } \quad u(0)=u \left ( \frac{2\pi}{\sqrt{p}} \right ) = 0.   
\end{equation}
The weak form is given by $\mathcal{F}(u,v) = \int_\Omega \nabla u \nabla v dx -\int_\Omega u^p v dx=0$ and $u = 0$ is an isolated solution in the sense of  \eqref{iso1}, namely, for any $w \in V$, we choose $v_w = w$, so that we have 
\begin{equation}
\left |\mathcal{F}(u + \varepsilon w, w)\right |= \left |\int_\Omega \nabla (\varepsilon w) \nabla w - \varepsilon^p \int_\Omega w^{p+1} \right | \geq \eta_u  \|\varepsilon w\|_{1,2}\|w\|_{1,2},
\end{equation}
for some $\eta_u$, no matter how $\varepsilon$ is small, due to Poincaré's inequality and Sobolev embedding, i.e., ($\|w\|_{0,p+1} \lesssim \|w\|_{1,2}$ in 1D). On the other hand, if we choose $w = \cos(\sqrt{p}x) - 1$, $0 < \varepsilon \ll 1$ and $v = \sin(\sqrt{p}x)$, then we have that 
\begin{equation}
\mathcal{F}(u+\varepsilon w,v)=\varepsilon \int_\Omega  \nabla w \nabla v dx-\varepsilon^p \int_\Omega  w^p v dx =-\varepsilon^p \int_\Omega  w^p v dx=0,      
\end{equation}
which implies that $0= \left |\mathcal{F}(u + \varepsilon w,v) \right | \geq \eta_u \varepsilon \|w\|_{1,2}\|v\|_{1,2}=O(\varepsilon)$. This will not make sense. 
\end{remark}

We begin with the following lemma as a consequence of our assumption:
\begin{lemma}\label{lem-ma} 
Under the assumption that $\|u_h\|_{0,\infty}, \|u\|_{0,\infty} \leq \Gamma$, we have 
\begin{equation} 
\|f^{(1)}\|_{0,\infty}, \|f^{(2)}\|_{0,\infty} < C (\Gamma),  
\end{equation} 
where $C$ is a constant that depends on $\Gamma$.  
\end{lemma}
We shall now consider the linearized problem for a given isolated solution $u$ to the equation \eqref{main:eq}: For $q \in V^*$, find $w \in V$ such that 
\begin{equation}
A(u;w,v) := a(w,v) + \int_\Omega f^{(1)}(u) w v \, dx =(q,v), \quad \forall v \in V. \label{linear_prob}
\end{equation}
This corresponds to the following partial differential equation: find $w$ such that 
\begin{equation}
-\Delta w + f^{(1)}(u) w = q \quad \mbox{ in } \Omega,  
\end{equation}
subject to the same type of boundary condition to the equation \eqref{main:eq}, but with $g$ replaced by zero function in \eqref{main:bdry}. We shall assume that the solution to the equation \eqref{linear_prob} satisfies the full elliptic regularity \cite{schatz1974observation}, i.e., 
\begin{equation}\label{ass2} 
\|w\|_{2,2} \lesssim \|q\|_{0,2}. 
\end{equation}
We shall now establish the well-posedness of the linearized problem as follows. 
\begin{lemma}\label{lem:1_2} 
$A(u;\cdot,\cdot)$, defined in Eq. (\ref{linear_prob}), satisfies the inf-sup condition, i.e., 
\begin{equation}
\inf_{w\in V} \sup_{v \in V} \frac{A(u;v,w)}{\|v\|_{1,2}\|w\|_{1,2}} = \inf_{v\in V} \sup_{w \in V} \frac{A(u;v,w)}{\|v\|_{1,2}\|w\|_{1,2}} \gtrsim 1.   \label{inf_sup_cont}
\end{equation}
\end{lemma}
\begin{proof}
Based on \cite{babuska1972survey}, we need to prove that
 \begin{itemize}
\item[(i)] there exists a unique zero solution, $w = 0$, to $A(u;w,v) = 0$ for all $v \in V$;    
\item[(ii)] $A(u;\cdot,\cdot)$ satisfies the Garding-type inequality, i.e.,  there exist $\gamma_0, \gamma_1 > 0$ such that 
\begin{equation*}
|A(u;v,v)| \geq \gamma_0 \|v\|_{1,2}^2 - \gamma_1 \|v\|_{0,2}^2, \quad \forall v \in V, 
\end{equation*}
\end{itemize}
For the second condition, the Garding-type inequality holds due to the Poincare inequality \cite{evans1990weak}. Secondly, we will prove the first condition using the proof by contradiction. Let us assume that there exists a non-zero solution  $w \in V$ such that 
\begin{equation}
A(u;w,v) = 0, \quad \forall v \in V.  
\end{equation}
Then, for $\epsilon$ sufficiently small, we define $u_\epsilon = u + \epsilon w \in N_{\delta}(u)$. Then, by Assumption \ref{assc}, we can choose $v_w \in V$, for which the inequality \eqref{iso1} holds and observe that 
\begin{eqnarray*}
\mathcal{F}(u + \epsilon w,v_w) &=& a(u + \epsilon w, v_w) + b(u + \epsilon w, v_w) = a(u + \epsilon w, v_w) \\
&& \quad  + \int_\Omega \left [ f(u) + f^{(1)}(u)(\epsilon w) + \frac{1}{2} f^{(2)}(\xi)(\epsilon w)^2 \right ]  v_w\, dx - \int_{\Gamma_{D_{_c}}} g v_w \, ds \\
&=& a(u,v_w) + b(u,v_w) + A(u,\epsilon w,v_w) + \frac{1}{2} \epsilon^2 \int_\Omega f^{(2)}(\xi) w^2 v_w \, dx \\
&\leq& \frac{\epsilon^2}{2} \|f^{(2)}\|_{0,p_1} \|w\|_{0,2p_2}^2 \|v_w\|_{0,p_3},  
\end{eqnarray*}
where the last inequality used the generalized H\"older inequality with $p_i$ for $i=1,2,3$, satisfying the identity
$
\frac{1}{p_1} + \frac{1}{p_2} + \frac{1}{p_3} = 1.$ Since we can choose $\epsilon$ to be arbitrarily small, this contradicts Assumption \ref{assc}. Thus, the proof is complete.

\end{proof}

After establishing the inf-sup condition for the linearized equation at the continuous level, we can establish the discrete inf-sup condition using the standard techniques such as the $W^{1,2}$ stability and $L^2$ norm error estimate of Ritz-projection for sufficiently small $h$ (see \cite{brenner2008mathematical,leykekhman2017maximum}).
\begin{lemma}
Under the Assumption \ref{assc}, the following discrete inf-sup condition holds if $h < h_0$ for sufficiently small $h_0$. Specifically, there exists $\alpha_0$, which is independent of $h$, such that  
\begin{equation}\label{infsupd} 
\inf_{w_h\in V_h} \sup_{v_h \in V_h} \frac{A(u;v_h,w_h)}{\|v_h\|_{1,2}\|w_h\|_{1,2}} = \inf_{v_h\in V_h} \sup_{w_h \in V_h} \frac{A(u;v_h,w_h)}{\|v_h\|_{1,2}\|w_h\|_{1,2}} = \alpha_0 > 0.  
\end{equation}
\end{lemma}
This finding has implications for the well-posedness of the Newton method used to find solutions. For a more in-depth discussion of using the Newton method to find multiple solutions,  see \cite{neuberger2001newton,rabinowitz1986minimax,wang2004local} and the references cited therein. We shall now consider the solution operator for $A(u;\cdot,\cdot)$ and the error estimates. First we define the projection operator $\pi_h : V \mapsto V_h$ as
\begin{equation}\label{solop} 
A(u;\pi_hw,v_h) = A(u;w,v_h), \quad \forall v_h \in V_h. 
\end{equation}
\begin{lemma}\label{lem:4_2} 
For the projection operator, we have with $r \geq 1$, 
\begin{equation}
\|w - \pi_hw\|_{0,2} \lesssim h^{r+1}\|w\|_{{r+1},2} \quad \mbox{ and } \quad \|w - \pi_hw\|_{1,2} \lesssim h^r \|w\|_{r+1,2}. 
\end{equation}
\end{lemma} 
The following inequality shall be needed for the well-posedness and error analysis of the discrete solution, which is well-known to be true, for the Dirichlet boundary condition case, \cite{brenner2008mathematical}.  
\begin{equation}\label{maxnorm1}
\|\pi_hw\|_{0,\infty} \lesssim \|w\|_{0,\infty}, \quad \forall w \in W^{1,2}(\Omega) \cap C^{1}(\Omega).  
\end{equation}
We note that for the mixed boundary condition, it can also be proven to be valid for the special case in 2D \cite{leykekhman2017maximum}, while the result in 3D, seems yet to be proven, while it has been proven to be valid for the pure Neumann bounary case in 3D recently in \cite{li2022maximum}. 
We can now establish that the discrete problem \eqref{main:heq} admits a unique solution that can approximate the fixed isolated solution $u$ with the desired convergence rate in both $L^2$ and $W^{1,2}$ norm, using the argument employed in \cite{xu1996two}. 
\begin{theorem}\label{thm} 
Suppose the Assumption \ref{assc} and \eqref{maxnorm1}. Then, for $h < h_0$, with $h_0$ sufficiently small, the finite element equation \eqref{main:heq} admits a solution $u_h$ satisfying 
\begin{equation}
\|u - u_h\|_{k,2} \lesssim h^{r+1-k}, \quad \mbox{ for } k = 0, 1.  
\end{equation}
Furthermore,  $u_h$ is the only solution in $N_\delta$ for some $\delta > 0$. 
\end{theorem}
\section{Companion-Based Multilevel finite element method (CBMFEM)}\label{CBMFEM}
In this section, we present a companion-based multilevel finite element method to solve the nonlinear system \eqref{goal1}. Solving this system directly is challenging due to the presence of multiple solutions. Therefore, we draw inspiration from the multigrid method discussed in \cite{brandt2011multigrid} that is designed for a single solution. We modify the multilevel finite element method by introducing a local nonlinear solver that computes the eigenvalues of the companion matrix. This enables us to generate a set of initial guesses for Newton's method, which is used to solve the nonlinear system on the refined mesh. 

\begin{figure}[!ht]
\centering
\includegraphics[width=0.25\textwidth]{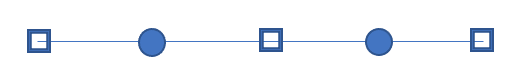}
\hspace{5mm} 
\includegraphics[width=0.25\textwidth]{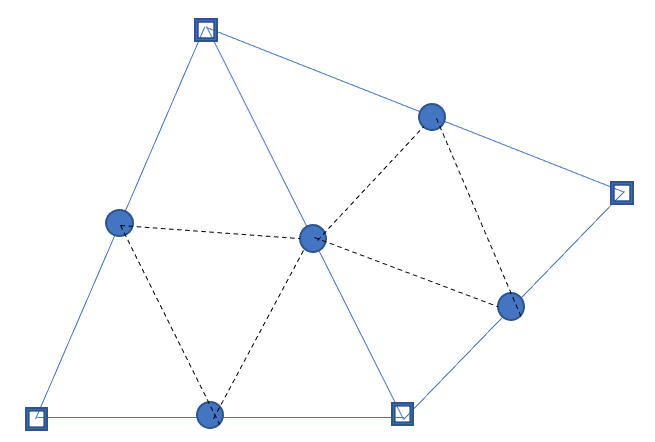}
\caption{Mesh refinement of CBMFEM in 1D (left) and 2D with edge (right). The square dots are the coarse nodes while filled circles are newly introduced fine nodes. }
\label{fig:mesh2}
\end{figure}

We begin by introducing a sequence of nested triangulations, namely,
\begin{equation}
\mathcal{T}_0, \mathcal{T}_1, \mathcal{T}_2, \cdots, \mathcal{T}_N,
\end{equation}
where $\mathcal{T}_0$ and $\mathcal{T}_N$ are the coarsest and the finest triangulations of $\Omega$, respectively. This leads to the construction of a sequence of nested and conforming finite element spaces $\{V_\ell\}_{\ell=0}^N \subset V$, given as follows: 
\begin{equation}
V_0 \subset \cdots \subset V_N. 
\end{equation}
The refinement strategy is shown in Figure \ref{fig:mesh2} for both 1D and 2D cases, where we introduce new nodes (the filled circles) based on the coarse nodes (the square dots).
Specifically, for a given coarse mesh $\mathcal{T}_H$, we obtain the refined mesh $\mathcal{T}_h$ by introducing a new node on $\mathcal{V}_H$. The overall flowchart of CBMFEM is summarized in Figure \ref{fig:mg}.
\begin{figure}[!ht]
\centering
\includegraphics[width=0.48\textwidth]{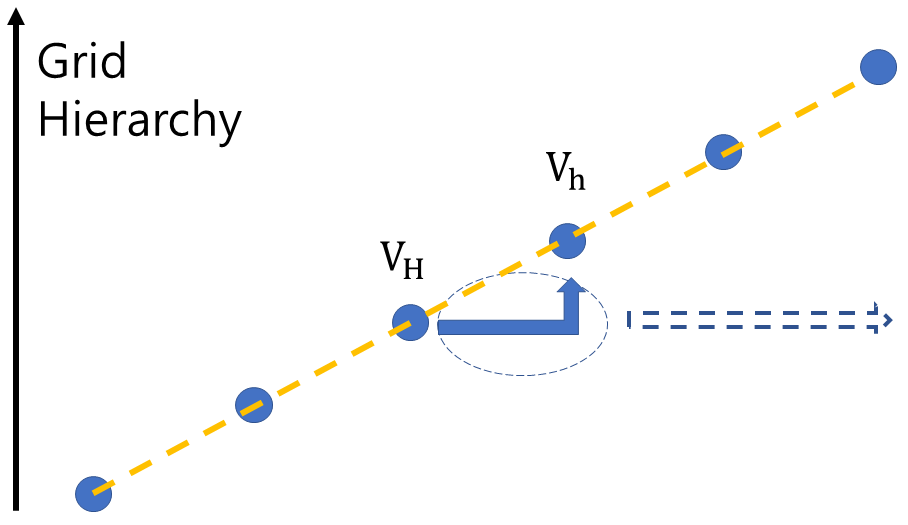}
\includegraphics[width=0.48\textwidth]{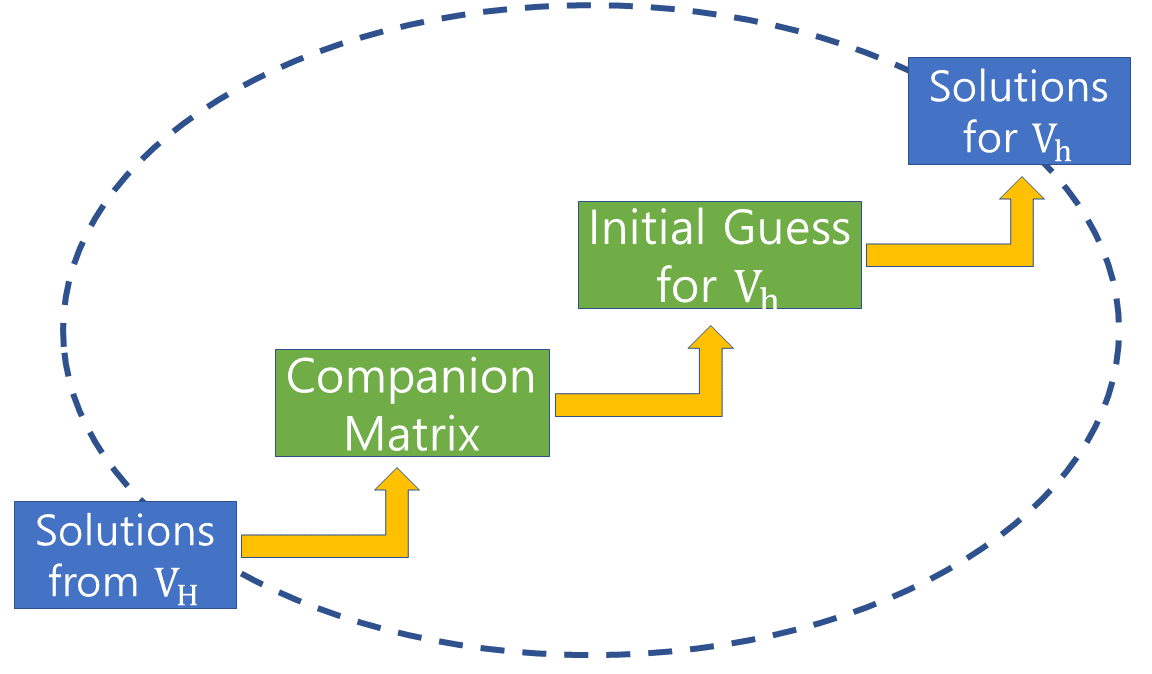}
\caption{A flowchart of the CBMFEM for solving the nonlinear differential equation. The hierarchical structure of CBMFEM is illustrated on the left, where for each level, we obtain a solution on the coarse grid, $V_H$. We then solve local nonlinear equations by constructing companion matrices and generate initial guesses for Newton's method on the finer level $V_h$ on the right.}
\label{fig:mg}
\end{figure}


Assuming that a solution $u_H \in V_H$ on a coarse mesh $\mathcal{T}_H$ has been well approximated, namely $\tenq{F}_H(\tenq{u}_H) = \tenq{0}$, we can refine the triangulation $\mathcal{T}_H$ to obtain a finer triangulation $\mathcal{T}_h$.
Since $\mathcal{V}_H \subset \mathcal{V}_h$, we can find a function $\widetilde{u}_h(x)$ in $\mathcal{V}_h$ such that we can express $u_H(x)$ as a linear combination of the basis functions of $\mathcal{V}_h$: 
\begin{equation}
\label{interr}
u_H(x) = \widetilde{u}_h(x)=\sum_{i=1}^{\dim{\mathcal{V}_h}} \widetilde{u}_h^i \phi_h^i(x),
\end{equation}
where $\dim{\mathcal{V}_h}$ is the dimension of $\mathcal{V}_h$ and $\phi_h^i(x)$ are the basis functions of $\mathcal{V}_h$. 
In particular, for any point $x_i \in \mathcal{V}_H$, we have $\widetilde{u}_h^i = u_H(x_i)$. For points $x_i \in \mathcal{V}_h \backslash \mathcal{V}_H$, we can calculate $\widetilde{u}_h^i$ based on the intrinsic structure of the basis functions. In other words, we can interpolate $u_H(x)$ to $u_h(x)$ by using the basis functions on two levels.  Specifically, let's consider 1D case and have the relation $\phi^{i/2}_{H}=\phi^{i/2}_{2h}=\frac{\phi^{i-1}_{h}}{2}+\phi^{i}_{h}+\frac{\phi^{i+1}_{h}}{2}$, which allows us to calculate the coefficients for (\ref{interr}). More precisely, we can write $u_H(x)$ to $u_h(x)$ as follows:
\begin{equation}
u_H(x)=\sum_{x_i \in V_H} {u}_H^{i/2} \phi_H^{i/2}(x)=\sum_{x_i \in V_H} {u}_H^{i/2} \left(\frac{\phi^{i-1}_{h}}{2}+\phi^{i}_{h}+\frac{\phi^{i+1}_{h}}{2}\right)=\sum_{i=1}^{\dim{\mathcal{V}_h}} \widetilde{u}_h^i \phi_h^i(x).
\end{equation}


We next update the value of $u_h^i$ for $x_i \in \mathcal{V}_h\backslash \mathcal{V}_H$ on the fine mesh. To do this, we solve $F_h^i(u_h^i;(u_h^j=\widetilde{u}_h^j)_{j\neq i}) = 0$ for $u_h^i$ by  fixing the values of other nodes as $\widetilde{u}_h$. Since $f(x,u)$ is a polynomial, we can rewrite $F_h^i$ as a single polynomial equation, namely,
\begin{equation}\label{polynomial}
F_h^i(u_h^i;(u_h^j=\widetilde{u}_h^j)_{j\neq i})=P_h^i(\widehat{u}_h^i)=0, 
\end{equation}
where $\displaystyle P_h^i(\alpha)=  \sum_{n=0}^mc_{n}\alpha^n $. The companion matrix of $P_h^i(\alpha)$ is defined as   
\begin{equation}\label{companion}
C(P_h^i) = \left[ \begin{array}{ccccc}
0 & 0 &  \ldots & 0& -c_0/c_{m}  \\
1 & 0 & \ldots & 0& -c_1/c_{m}  \\
0 & 1 &  \ldots & 0& -c_2/c_{m}  \\ 
\vdots & \vdots &\ddots & \vdots  &  \vdots \\ 
0 & 0 & \ldots & 1& -c_{m-1}/c_{m}
\end{array} \right] ,
\end{equation}
where, except for $c_0$, the coefficients $c_n$ depend only on  $\{\widetilde{u}_h^j\}$ that are near the point  $x_i \in \mathcal{V}_h\backslash \mathcal{V}_H$ , making their computation local. By denoting the root of Eq. (\ref{polynomial} ) as  $\widehat{u}^i_h $, the initial guess for the solutions on $V_h$ is set as 
\begin{equation}
\label{initialguess}
\widehat{{u}}_h=  \sum \widehat{{u}}_h^i \phi^i_h \hbox{~where~} \widehat{{u}}_h^i = \left\{\begin{array}{rl}u^i_H &  x_i \in \mathcal{V}_H,\\  \widehat{u}^i_h \hbox{~by solving~} (\ref{polynomial})  &  x_i \in \mathcal{V}_h\backslash \mathcal{V}_H.\end{array} \right.
\end{equation}
Since all the eigenvalues of $C(P_h^i)$ satisfy the equation $P_h^i(y_i) = 0$, there can be up to $m^{|\mathcal{V}_h\backslash\mathcal{V}_H|}$ possible initial guesses, where $|\mathcal{V}_h\backslash\mathcal{V}_H|$ denotes the number of newly introduced fine nodes on $\mathcal{V}_h$ and $m$ is the degree of the polynomial \eqref{polynomial}. However, computing all of these possibilities is computationally expensive, so we apply the filtering conditions below to reduce the number of initial guesses and speed up the method:
\begin{itemize}[leftmargin=*]
\item {\bf Locality condition:} we assume the initial guess is  near $\tenq{\tilde{u}_h}$ in term of the residual, namely,
\begin{equation}
    \|\tenq{F}_h(\tenq{\widehat{u}_h})\|_{0,2}<C_1\|\tenq{F}_h(\tenq{\tilde{u}_h})\|_{0,2};
\end{equation}
\item {\bf Convergence condition:} we apply the convergence estimate to the initial guess, namely, \begin{equation}
\|\tenq{F}_h(\tenq{\widehat{u}_h})\|_{0,2}<C_2h^2;
\end{equation}
\item {\bf Boundness condition:}  we assume the initial guess is bounded, namely, \begin{equation}
\|\tenq{\widehat{u}_h}\|_{0,\infty}<C_3.
\end{equation}
\end{itemize}
Finally, we summarize the algorithm of CBMFEM  in {\bf Algorithm \ref{alg:main}}.

\begin{algorithm}
\caption{CBMFEM for computing multiple solutions}\label{alg:main}
Given $V_h, V_H$, and solution $u_H$ on $V_H$.

\begin{algorithmic}

\State{Interpolate $u_H = \sum \widetilde{u}_h^i \phi_h^i$ and compute coefficient $\widetilde{u}_h^i$.}
\For{$i \in V_h\backslash V_{H}$}

\State{Construct the polynomial equation $P_h^i(\widehat{\tenq{u}}_h^i)$}
\State{Compute the eigenvalues of the companion matrix, $C(P_h^i)$} 

\EndFor

\State{Obtain initial guesses in (\ref{initialguess}) on $V_h$ and apply filtering conditions.}

\State{Employ Newton method on $V_h$ with the obtained initial guesses to solve $\tenq F_h(\tenq{u}_h) = 0$}

\end{algorithmic}
\end{algorithm}

\section{Numerical Examples}\label{num} 
In this section, we present several examples for both 1D and 2D to demonstrate the effectivity and robustness of CBMFEM with $r = 1$ for simplicity. We shall let $e_H^h = u_h-u_H$ where $u_h$ and $u_H$ are the numerical solutions with grid step size $h$ and $H$, respectively, and $u$ is the analytical solution. For error analysis, since most of the examples we considered do not have analytical solutions, we used asymptotic error analysis to calculate the convergence rate.

\subsection{Examples for 1D} 

\subsubsection{Example 1}
First, we consider the following boundary value problem
\begin{equation}\label{ex2}
\begin{cases} -u''=(1+u^4) \quad \text{on } \quad[0,1], \\u'(0)=u(1)=0,\end{cases}
\end{equation}
which has analytical solutions \cite{hao2014bootstrapping}. More specifically, by multiplying both side with $u'$ and integrating with respect to $x$,  we obtain
\begin{equation}
\frac{(u'(x))^2}{2}+F(u(x))-F(u_0)=0,
\end{equation}
where $
F(u)=u+\frac{u^5}{5}
$ and 
$u_0=u(0)$. Since $u'(0)=0$ and $u''(x) <0$, then we have
$u'<0$ for all $x>0$. Moreover,  $u(x)>u_0$ implies $F(u_0)>F(x)$ for all $x>0$. 
Therefore
\begin{equation}
     \frac{u'(x)}{\sqrt{F(u_0)-F(u(x))}}=-\sqrt{2}.
\end{equation}
By integrating $x$ from $s$ to $1$, we obtain
\begin{equation}
\int_s^1     \frac{u'(x)}{\sqrt{F(u_0)-F(u(x))}}dx=-\sqrt{2}(1-s).
\end{equation}
 Due to the boundary condition $u(1)=0$, we have
\begin{equation}
    \int_0^{u(s)} \frac{dx}{\sqrt{F(u_0)-F(u(x))}}=\sqrt{2}(1-s).
\end{equation}
By choosing $s=0$, we have the following equation for $u_0$,
\begin{equation}\label{analytic_result}
    \int_0^{u_0} \frac{dx}{\sqrt{F(u_0)-F(u(x))}}=\sqrt{2}
\end{equation}
Then for any given $u_0$, the solution of (\ref{ex2}) is uniquely determined  by the initial value problem 
\begin{equation}
\begin{cases} - v'=(1+u^4)  \\u'=v\end{cases}\hbox{~with~}\begin{cases}v(0)=0\\u(0)=u_0\end{cases}.
\end{equation} 
By solving (\ref{analytic_result}) with Newton's method, we get two solutions $u_0\approx0.5227$ and $u_0\approx1.3084$. 
Then the numerical error  is shown in Table \ref{ex1_MTG} for the CBMEFM with Newton's nonlinear solver.

\begin{figure}[ht!]
\centering
\includegraphics[width=14cm,height=6cm,keepaspectratio]{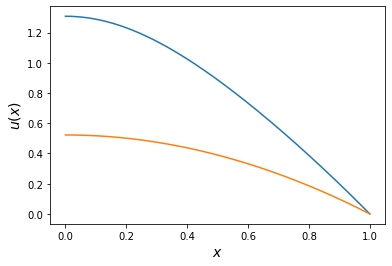}
   \caption{Numerical solutions of Eq. (\ref{ex2}) with $N=1025$ grid points. }
   \label{1d}
\end{figure}

\begin{table}[ht!]
\centering\footnotesize
\begin{tabular}{||c|c|c|c|c|c|c|c|c|c|c|c|}
\hline
  \multirow{2}{1.5em}{h}& \multicolumn{4}{c|}{\# 1st Solution} &\multicolumn{4}{c|}{\# 2nd solution}&\multicolumn{1}{c||}{CPU(s)} \\
\cline{2-10}
& $\|e_h\|_{0,2}$&Order & $\|e_h\|_{1,2}$ & Order  & $\|e_h\|_{0,2}$& Order & $\|e_h\|_{1,2}$ & Order &Newton\\
\hline
$2^{-2}$ &2.6E-04  & x   &1.7E-01 & x    & 8.3E-03 &    x  & 7.0E-02 & x   & 0.10 \\
$2^{-3}$ &6.4E-05  &2.01 &8.0E-02 & 1.00 & 2.0E-03 & 2.04 & 4.0E-02 &1.01   & 0.10 \\
$2^{-4}$ &1.6E-05  &2.00 &4.0E-02 & 1.00 & 5.0E-04 & 2.01 & 2.0E-02 &1.00 & 0.15 \\
$2^{-5}$ &4.0E-06  &2.00 &2.0E-02 & 1.00 & 1.3E-04 & 2.00 & 9.3E-03 &1.00 & 0.16 \\
$2^{-6}$ &1.0E-06  &2.00 &1.0E-02 & 1.00 & 3.1E-05 & 2.00 & 4.7E-03 &1.00 & 0.11 \\
$2^{-7}$ &2.5E-07  &2.00 &5.2E-03 & 1.00 & 7.8E-06 & 2.00 & 2.3E-03 &1.00 &0.16 \\
$2^{-8}$ &6.2E-08  &2.00 &2.6E-03 & 1.01 & 1.9E-06 & 2.00 & 1.2E-03 &1.01 & .25 \\
$2^{-9}$ &1.6E-08  &2.00 &1.3E-03 & 1.03 & 4.9E-07 & 2.00 & 0.6E-03 &1.04 &  0.64 \\
$2^{-10}$&3.9E-09  &2.00 &5.6E-04 & 1.16 & 1.2E-07 & 2.00 & 0.3E-03 &1.16 &  3.55 \\

\hline
\end{tabular}
\caption{Numerical errors and computing time of CBMEFM with Newton's nonlinear solver for solving Eq. (\ref{ex2}).} \label{ex1_MTG}
\end{table}

\subsubsection{Example 2}
Next, we consider the following boundary value problem
\begin{equation}\label{example2}
\begin{cases} - u''=-u^2 \quad \text{on } \quad [0,1], \\u(0)=0\hbox{~and~}u(1)=1,\end{cases}
\end{equation}
which has two solutions shown in  Fig. \ref{Fig1}. We start $N=3$ and compute the solutions up to $N=1025$ by implementing CBMEFM with nonlinear solver. We compute the numerical error by using $\|e_H^h\|_{0,2}$ and summarize the convergence test and computing time
in Table \ref{ex22 MTG}. 

\begin{figure}[ht!]
\centering
\includegraphics[width=14cm,height=6cm,keepaspectratio]{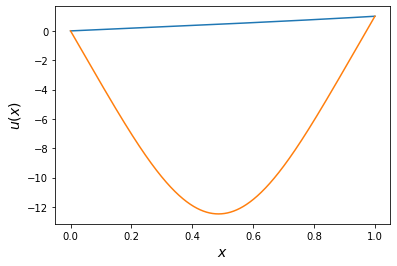}
   \caption{Numerical solutions of Eq. (\ref{example2}) with $N=1025$ grid points.}\label{Fig1}
\end{figure}



\begin{table}[ht!]
\centering \footnotesize
\begin{tabular}{||c| c| c |c| c|  c| c| c| c| c| c||}

\hline
  \multirow{2}{1.5em}{h}& \multicolumn{4}{c|}{\# 1st solution} &\multicolumn{4}{c|}{\# 2nd solution}&\multicolumn{1}{c||}{CPUs} \\
\cline{2-10}
& $\|e_H^h\|_{0,2}$&Order & $\|e_H^h\|_{1,2}$ & Order  & $\|e_H^h\|_{0,2}$& Order & $\|e_H^h\|_{1,2}$ & Order & Newton\\
\hline

$2^{-2}$      &2.6E-3& x&2.3E-2& x&7.9E-1 &x & 5.8E-0 &x&0.09\\
$2^{-3}$      &6.8E-4&1.96 &1.2E-2&0.90& 2.0E-1&2.01&2.9E-0&1.03& 0.09\\
$2^{-4}$      & 1.7E-4&1.99&6.3E-3&0.94& 4.9E-2&  2.00  &1.4E-0&1.00 &0.09\\
$2^{-5}$      &4.3E-5& 2.00  &3.2E-3&0.97&  1.2E-2&2.00 & 7.12E-1&1.00 & 0.10\\
$2^{-6}$      &1.1E-5&2.00&1.6E-3&0.98& 3.1E-3&2.00 &3.6E-1&1.00& 0.11  \\
$2^{-7}$      &2.7E-6&2.00&8.2E-4&0.99& 7.7E-4&2.00 &1.8E-1&1.00& 0.14  \\
$2^{-8}$      &6.7E-7&2.00&4.1E-4&1.00& 1.9E-4&2.00 & 8.9E-2&1.00& 0.24   \\
$2^{-9}$      &1.7E-7&2.00 &2.1E-4&1.00& 4.8E-5&2.00  &4.5E-2&1.00& 0.49   \\
$2^{-10}$     &4.2E-8&2.00&1.0E-4&1.00& 1.2E-5& 2.00 &2.2E-2&1.00   & 3.35  \\
\hline
\end{tabular}
\caption{\label{ex22 MTG}Numerical errors and computing time of CBMEFM with Newton's solver for solving Eq. (\ref{example2}).}
\end{table}

\subsubsection{Example 3}
Thirdly, we consider the following nonlinear parametric differential equation
\begin{equation}\label{ex3}
\begin{cases} - u''=u^2(p-u^2) \quad \text{on } \quad [0,1], \\u'(0)=u(1)=0,\end{cases}
\end{equation}
where $p$ is a parameter. The number of solutions increases as $p$ gets larger \cite{hao2020adaptive}.  We compute the numerical solutions for $p=1,7$, and $18$ using CBMEFM with Newton's solver and show the solutions in Fig. \ref{1dp1}. he computation time and the number of solutions for different values of $p$ and step sizes are summarized in Table \ref{ex3 NM}. As $p$ increases, the number of solutions increases, and hence the computation time becomes longer.



\begin{figure}[!ht]
\centering
\includegraphics[width=4.3cm]{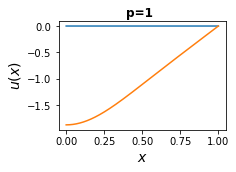}
\includegraphics[width=4.1cm]{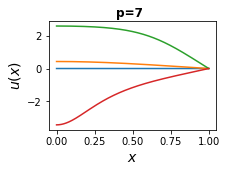}
\includegraphics[width=4.2cm]{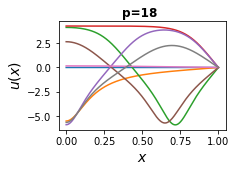}
   \caption{Numerical solutions of Eq. (\ref{ex3}) with $1025$ grid points for $p=1$, $p=7$, and $p=18$, respectively. }
   \label{1dp1}
\end{figure}

\begin{table}[!ht] 
    \centering \footnotesize
\begin{tabular}{||c| c| c| c| c| c| c||}
\hline
  \multirow{2}{1.5em}{ h }& \multicolumn{2}{c|}{p=1} &\multicolumn{2}{c|}{p=7}&\multicolumn{2}{c||}{p=18} \\
\cline{2-7}
& CPUs  & \# of sols  & CPUs  & \# of sols & CPUs  & \# of sols \\
\hline
$2^{-2}$      &0.17&2 & 0.15&5 & 0.29 & 19 \\
$2^{-3}$      &0.12&2& 0.35 &9& 3.01  &37\\
$2^{-4}$      & 0.11&2 & 0.20&4 & 28.17 &10 \\
$2^{-5}$      & 0.11&2& 0.11&4 & 0.38 &8 \\
$2^{-6}$      &0.11&2& 0.14&4 & 0.18  &8\\
$2^{-7}$      &0.14&2& 0.23&4 & 0.33  &8\\
$2^{-8}$      &0.26&2& 0.35&4 & 0.66  &8 \\
$2^{-9}$      &0.57&2& 0.95&4 & 1.93  &8 \\
$2^{-10}$     &3.08&2 & 5.55&4 & 11.54  &8 \\
\hline
\end{tabular}

\caption{Computing time (in seconds) and the number of solutions of Eq. (\ref{ex3}) for different $p$ by CBMEFM with Newton's solver. \label{ex3 NM}}
\end{table}

\begin{table}[!ht] 
\footnotesize
\begin{tabular}{||c| c| c| c| c| c| c|c| c|c| c|c| c||}
\hline
  \multirow{2}{1.5em}{h}& \multicolumn{4}{c|}{\# 1st solution } &\multicolumn{4}{c|}{\# 2nd solution}&\multicolumn{4}{c||}{\# 3rd solution} \\
\cline{2-13}
&$\|e_H^h\|_{0,2}$& Order &$\|e_H^h\|_{1,2}$ &Order  & $\|e_H^h\|_{0,2}$& Order  & $\|e_H^h\|_{1,2}$&Order & $\|e_H^h\|_{0,2}$& Order  &$\|e_H^h\|_{1,2}$&Order\\
\hline
$2^{-2}$      &1.6E-1&x&1.43& x& 1.6E-1&x&9.0E-1&x  & 2.7E-2 &x&1.1E-1&x\\
$2^{-3}$      &2.4E-1&-0.54&1.22&0.23& 3.3E-2& 2.30 &4.7E-1 & 0.93& 6.8E-3&  2.00  &5.1E-2&1.08\\
$2^{-4}$      & 4.2E-2&2.52&7.4E-1&0.73 & 7.4E-3&2.13&2.2E-1 & 1.07 &1.7E-3& 2.00 &2.5E-2&1.06 \\
$2^{-5}$      & 8.4E-3& 2.31 &3.0E-1&1.28& 1.9E-3&2.01&1.1E-1 & 1.01 & 4.3E-4&  2.00 &1.2E-2&1.03 \\
$2^{-6}$      &2.1E-3&2.04&1.4E-1&1.08& 4.6E-4&2.00&5.5E-2 & 1.00 & 1.1E-4& 2.00 &5.9E-3&1.02\\
$2^{-7}$      &5.1E-4&2.01&7.0E-2&1.03& 1.2E-4&2.00&2.8E-2 & 1.00 & 2.7E-5& 2.00 &3.0E-3&1.01\\
$2^{-8}$      &1.3E-4&2.00&3.5E-2&1.02& 2.9E-5&2.00&1.4E-2 & 1.00 & 6.7E-6& 2.00 &1.5E-3&1.00 \\
$2^{-9}$      &3.2E-5&2.00&1.7E-2&1.01& 7.2E-6&2.00&6.9E-3 & 1.00 & 1.7E-6& 2.00 &7.4E-4&1.00 \\
$2^{-10}$     &8.0E-6&2.00&8.6E-3&1.00 & 1.8E-6&2.00&3.5E-3 & 1.00 & 4.2E-7& 2.00 &3.7E-4&1.00 \\
\hline
\end{tabular}

\caption{Numerical errors only nontrivial solutions for Eq. (\ref{ex3}) when $p=7$ by CBMEFM with Newton's solver. }
\end{table}

\subsubsection{Example 4}
We consider the following semi-linear elliptic boundary value problem
\begin{equation}\label{exnew}
\begin{cases} -u''(x)=d^{r-2}|x|^r u^3(x)  \quad \text{on}\quad  [-1,1], \\u(-1)=u(1)=0 ,
\end{cases}
\end{equation}
where $d$ is the scaling coefficient corresponding to the domain. This example is based on a problem considered in \cite{xie2012finding,xie2022solving}. We have re-scaled the domain from $[-1,1]$ to $[-d,d]$ and formulated the problem accordingly.

We start with $r=3$ and $d=1$ and find $4$ non-negative solutions with $N=1025$ using CBMEFM. Using the homotopy method with respect to both $r$ and $d$, we also discover the same number of solutions for $r=3, d=10$ and $r=12, d=1$. The numerical solutions with different parameters are shown in Fig. \ref{newexample}. Then, we also create a bifurcation diagram for the numerical solutions with respect to $r$ by choosing $d=1$ and using $1025$ grid points, as shown in Fig. \ref{newexample_bif}. We only display the non-trivial non-negative solutions on this diagram. When $r$ is small, we have only one non-trivial solution. As $r$ increases, the number of non-trivial solutions increases and bifurcates to three solutions.

\begin{figure}[!ht]
\centering
\includegraphics[width=3.7cm]{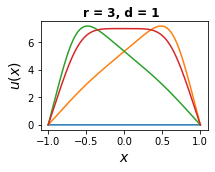}
\includegraphics[width=3.9cm]{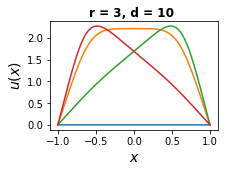}
\includegraphics[width=3.8cm]{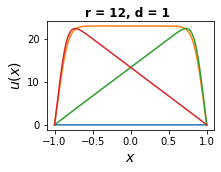}
   \caption{Numerical solutions of Eq. (\ref{exnew}) with $N=1025$ grid points for different $r$ and $d$.
   We have symmetric solutions so we only concern with one of them.}
   \label{newexample}
\end{figure}

\begin{figure}[ht!]
\centering
\includegraphics[width=14cm,height=6cm,keepaspectratio]{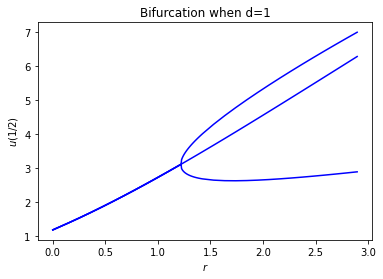}
   \caption{Bifurcation diagram of Eq. (\ref{exnew}) with respect to $r$ with $N=1025$ grid points and $d=1$.}
      \label{newexample_bif}
\end{figure}

\begin{table}[!ht] 
\centering
\footnotesize
\begin{tabular}{||c| c| c| c| c| c| c|c| c||}
\hline

  \multirow{2}{1.5em}{h} & \multicolumn{4}{c|}{\# 1st solution } & \multicolumn{4}{c||}{\# 2nd solution} \\
  
\cline{2-9}
&$\|e_H^h\|_{0,2}$& Order  &$\|e_H^h\|_{1,2}$&Order  & $\|e_H^h\|_{0,2}$& Order  & $\|e_H^h\|_{1,2}$&Order \\
\hline
$2^{-3}$      &1.3E-0& x &  7.9E-0&x & 1.0E-0&x&9.2E-0&x   \\
$2^{-4}$      &3.2E-1&2.04&4.2E-0&0.90& 5.7E-1&  0.87&5.4E-0 & 0.76\\
$2^{-5}$      & 8.0E-2&2.03&2.0E-0&1.06 & 1.1E-1&2.33&2.3E-0 & 1.22  \\
$2^{-6}$      & 2.0E-2&2.03&1.0E-0&1.01& 2.7E-2&2.03&1.2E-0 & 1.02 \\
$2^{-7}$      &4.9E-3&2.01&5.0E-1&1.00 & 6.7E-3&2.00&5.7E-1 & 1.00\\
$2^{-8}$      &1.2E-3&2.00&2.5E-1&1.00& 1.7E-3&2.00&2.9E-1 & 1.00 \\
$2^{-9}$      &3.1E-4&2.00&1.3E-1&1.00& 4.2E-4&2.00&1.4E-1 & 1.00 \\
$2^{-10}$      &7.6E-5&2.00&6.3E-2&1.00& 1.0E-4&2.00&7.2E-2 & 1.00  \\

\hline
\end{tabular}

\caption{Numerical errors only nontrivial solutions for Eq. (\ref{exnew}) when $r=3, d=1$ by CBMEFM with Newton's solver. }
\end{table}

\subsubsection{The Schnakenberg model}
We consider the steady-state  system of the Schnakenberg model in 1D with no-flux boundary conditions
\begin{equation}
\begin{cases}  u''+\eta(a-u+u^2v)=0, \quad \text{on } \quad [0,1] \\\ d v''+\eta(b-u^2v)=0, \quad  \quad \text{on } \quad [0,1] \\\ u'(0)=u'(1)=v'(0)=v'(1)=0. \end{cases}\label{Sch_equ}
\end{equation}
This model exhibits complex solution patterns for different parameters $\eta, a, b,$ and $d$ \cite{hao2020spatial}.
 Since there is only one nonlinear term $u^2v$, we rewrite the steady-state system as follows
\begin{equation}
\begin{cases}  u''+d v''+\eta(a+b-u)=0, \\d v''+\eta(b-u^2v)=0.\end{cases}\label{changed_from}
\end{equation}
In this case, we only have one nonlinear equation in the system (\ref{changed_from}). After discretization, we can solve ${\tenq{v}}_h$ in terms of ${\tenq{u}}_h$ using the first equation and plug it into the second equation. We obtain a single polynomial equation of ${\tenq{u}}_h$, which allows us to use the companion matrix to solve for the roots. Then we use CBMEFM to solve the Schakenberg model in 1D with the parameters $a=1/3, b=2/3, d=50, \eta =50$ up to $h=2^{-9}$. There are 3 solutions computed (\ref{ex5sol}) and we show the computing time in Table (\ref{ex4 NM}).

\begin{table}[ht!]
\small
\centering
\begin{tabular}{||c| c| c| c| c| c||}
\hline
  \multirow{2}{1.5em}{h}& \multicolumn{4}{c|}{\# 1st solution} & CPUs\\
\cline{2-5}
& $\|e_H^h\|_{0,2}$& Order&$\|e_H^h\|_{1,2}$ &Order &    \\
\hline

$2^{-3}$      &2.9E-2&x &1.9E-1&x& 0.58 \\
$2^{-4}$      & 6.3E-3&2.20&6.0E-2&1.63& 0.42\\
$2^{-5}$      &1.6E-3& 2.02  &2.7E-2&1.18& 0.55 \\
$2^{-6}$      &3.9E-4&2.00&1.3E-2&1.06& 0.84   \\
$2^{-7}$      &9.6E-5&2.00&6.3E-3&1.02& 1.96   \\
$2^{-8}$      &2.4E-5&2.00&3.1E-3&1.01& 6.14    \\
$2^{-9}$      &6.0E-6&2.00 &1.6E-3&1.00& 22.14    \\
$2^{-10}$     &1.5E-6&2.00&7.8E-4&1.00   & 157.88  \\
\hline
\end{tabular}
\caption{\label{ex4 NM}Numerical errors and computing time of CBMEFM with Newton's solvers for solving Eq. (\ref{Sch_equ}). We have symmetric solutions so we only concern with one of them.}
\end{table}

\begin{figure}[!ht]
\centering
\includegraphics[width=5cm]{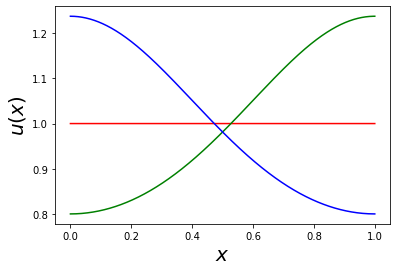}
\includegraphics[width=5cm]{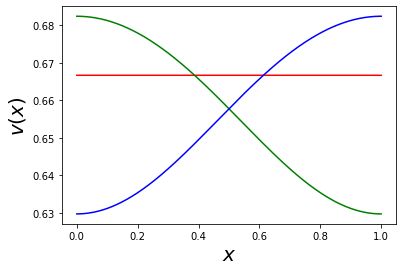}
  \caption{$3$ different solutions for equation  (\ref{Sch_equ}) with $N=1025$ grid points. The same colors are paired solutions.}
  \label{ex5sol}
\end{figure}

\subsection{Examples in 2D}
In this section, we discuss a couple of two dimensional examples. 

\subsubsection{Example 1} 
\begin{equation}\label{2dex}
\begin{cases} \Delta u(x,y)+u^2(x,y)=s\sin(\pi x)\sin(\pi y)  \quad \text{on} \quad \Omega, \\u(x,y)=0,\quad  \text{on} \quad \partial \Omega\end{cases} 
\end{equation}
where $\Omega=(0,1) \times (0,1)$ \cite{breuer2003multiple}. 


In this example, we utilized edge refinement to generate a multi-level grid and provide multiple initial guesses for the next level, as shown in Fig. \ref{Fig:Edge}. The number of nodes and triangles on each level is summarized in Table \ref{2d:MG}.
\begin{figure}[ht!]
\centering
\includegraphics[width=10cm]{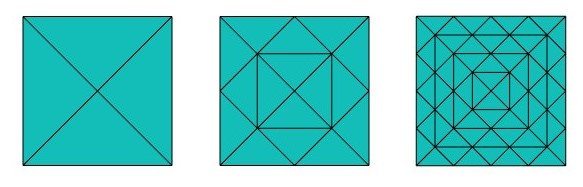}
  \caption{Multi-level grid based on the edge refinement with a rectangular domain.}\label{Fig:Edge}
\end{figure}

\begin{table}[ht!]
    \centering
\begin{tabular}{|l*{9}{|c}r}
\hline
{Step size} & $2^{0}$ & $2^{-1}$& $2^{-2}$& $2^{-3}$& $2^{-4}$& $2^{-5}$& $2^{-6}$ & $2^{-7}$\\
\hline

{\# of Nodes} & $5$ & $13$& $41$& $145$& $545$& $2113$& $8321$& $33025$\\
\hline

{\# of Triangles} & $4$ & $16$& $64$& $256$& $1024$& $4096$& $16384$ & $65536$\\
\hline
\end{tabular}
\caption{The number of nodes and triangles of the multigrid with the edge refinement for each level.} \label{2d:MG}
\end{table}

First, we compute the solutions with $s=1600$. We only applied the CBMEFM for the first  three multi-level grids until $\ell=2$ due to the extensive computation caused by a large number of solution combinations on the higher level.  Starting $\ell=3$, we used Newton refinement with an interpolation initial guess from the coarse grid.
We found 10 solutions and computed them until $h=2^{-7}$.
Since some solutions are the same up to the rotation, we plot only four solutions in Fig. \ref{2dbif}. It is worth noting that Eq. (\ref{2dex}) remains unchanged even when $x$ is replaced by $1-x$, and similarly for $y$. Therefore, rotating solutions 3 and 4 from Fig. \ref{2dgraph} would also yield valid solutions. Therefore, although there are 10 solutions in total,  we only plot the 4 solutions in Fig. \ref{2dgraph}. up to the rotation. 
We also computed the numerical error and  convergence order in both $L^2$  and $H^1$ norms for these solutions and summarized them in Table \ref{2dex:L2}.

\begin{table}[ht!]
    \centering
\begin{tabular}{l*{6}{c}r}
&Method&Step size &  \# of solutions             & Comp. Time  \\
\hline
&\multirow{3}{4em}{CBMEFM}&$2^{0}$ & 2 &  0.2s   \\
&&$2^{-1}$ &10     & 1.4s  \\
&&$2^{-2}$ &10       &  142.7s  \\
\hline
&\multirow{5}{4em}{Newton's refinement }&$2^{-3}$ &10  &  7.1s   \\
&&$2^{-4}$ &10   &  20.6s   \\
&&$2^{-5}$ &10 &  76.8s  \\
&&$2^{-6}$ & 10     &  540.4s   \\
&&$2^{-7}$ & 10     &  5521.3s   \\
\end{tabular}
\caption{The numerical performance summary of solving (\ref{2dex}) with $s=1600$. Until the level $\ell=2$ we used CBMEFM and applied Newton's refinement starting from  $\ell=3$.}
\end{table}

\begin{figure}[!ht]
\centering
\includegraphics[width=13.cm,height=3cm]{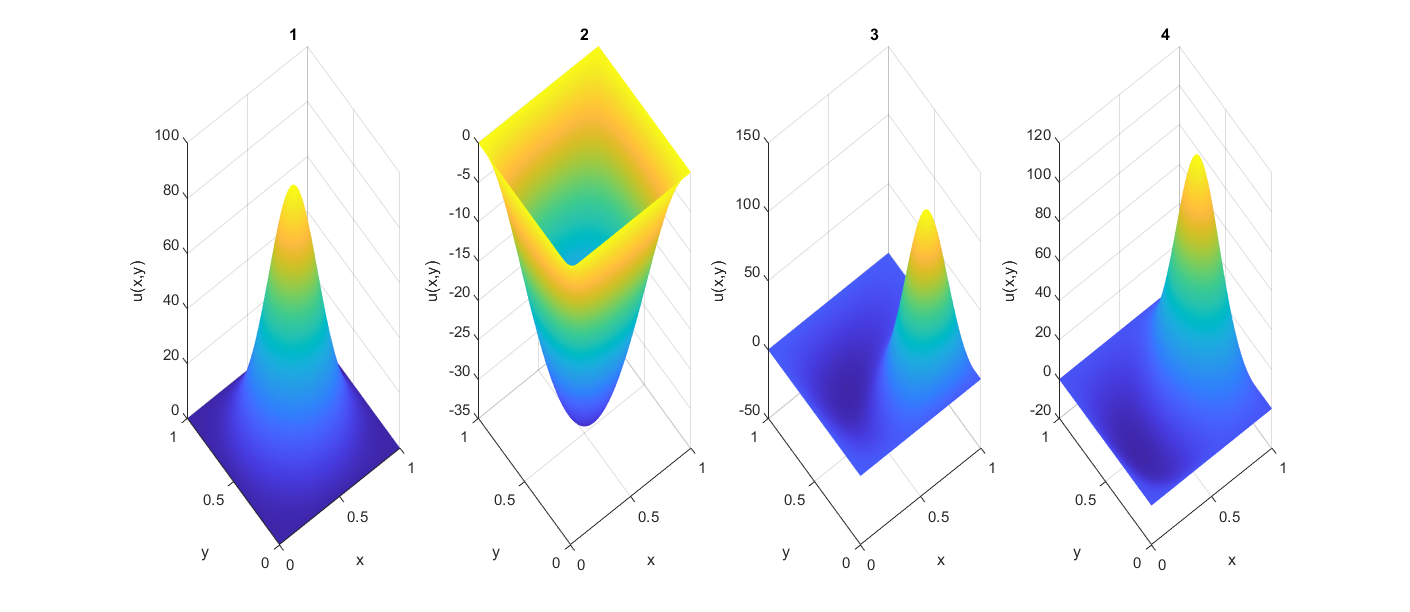}
  \caption{Multiple solutions of Eq. (\ref{2dex}) with $s=1600$ and a step size of $2^{-7}$.}
  \label{2dgraph}
\end{figure}

\begin{table}[!ht]
\centering\small
\begin{tabular}{||c| c| c| c| c| c| c|c| c||}
\hline
  \multirow{2}{1.5em}{h}& \multicolumn{2}{c|}{\# 1st solution} &\multicolumn{2}{c|}{\# 2nd solution}&\multicolumn{2}{c|}{\# 3rd solution}&\multicolumn{2}{c||}{\# 4th solution}  \\
\cline{2-9}
& $\|e_{_H}^h\|_{0,2}$ & Order  & $\|e_{_H}^h\|_{0,2}$ & Order &  $\|e_{_H}^h\|_{0,2}$ & Order&  $\|e_{_H}^h\|_{0,2}$ & Order\\
\hline

$2^{-1}$      &2.1E-0 &x& 1.2E-0 &x &x &x &x  &x \\
$2^{-2}$      &5.1E-0 & -1.26& 2.8E-0& -1.26& 2.3E+1 &x &5.4E-0 &x \\
$2^{-3}$      &3.9E-0 & 0.37& 0.7E-0& 1.94& 9.5E-0 & 1.28&6.9E-0 & -0.36\\
$2^{-4}$      &1.6E-0 & 1.30 & 0.2E-0& 1.98& 3.2E-0 & 1.56&2.3E-0 & 1.58\\
$2^{-5}$      &0.5E-0 & 1.70& 0.5E-1& 1.99 & 0.9E-0 & 1.85&0.7E-0 & 1.83 \\
$2^{-6}$      &0.1E-0 & 1.90& 0.1E-1& 2.00 & 0.2E-0 & 1.95&0.2E-0 & 1.95  \\
$2^{-7}$      &3.3 E-3 & 1.97& 2.9E-03& 2.00& 0.6E-1 & 1.99&0.4E-1 &1.99  \\
\hline
h& $\|e_H^h\|_{1,2}$ & Order  & $\|e_H^h\|_{1,2}$ & Order &  $\|e_H^h\|_{1,2}$ & Order&  $\|e_H^h\|_{1,2}$ & Order\\
\hline

$2^{-1}$      &2.6E+1 &x& 1.4E+1 &x &x & x& x & x\\
$2^{-2}$      &6.8E+1 & -1.40& 2.7E+1& -0.93& 1.9E+2 & x&7.2E+1 &x \\
$2^{-3}$      &5.6E+1 & 0.29& 1.5E+1& 0.91& 9.9E+1 &  0.92&7.9E+1 & -0.13\\
$2^{-4}$      &3.1E+1 & 0.86& 7.4E-0&  0.97 & 5.2E+1 & 0.92&4.2E+1 & 0.92\\
$2^{-5}$      &1.6E+1 & 0.97& 3.7E-0&0.99& 2.6E+1 & 1.02 &2.1E+1 & 1.01 \\
$2^{-6}$      &7.8E-0 & 1.00& 1.9E-0& 1.00& 1.3E+1 & 1.01&1.0E+1 & 1.01  \\
$2^{-7}$      &3.9E-0 & 1.00& 0.9E-0& 1.00& 6.3E-0 & 1.00&5.1E-0 & 1.00 \\
\hline

\end{tabular}
\caption{Numerical $L^2$ and $H^1$ errors for multiple solutions of Eq. (\ref{2dex}) with $s=1600$ shown in Fig. \ref{2dgraph}.}\label{2dex:L2}
\end{table}

Finally, we also explored the solution structure with respect to 
For small values of $s$, only two solutions are observed. However, as $s$ increases, the number of solutions also increases, as demonstrated in Fig. \ref{2dbif}.

\begin{figure}[ht!]
\centering
\includegraphics[width=10.cm,height=5cm]{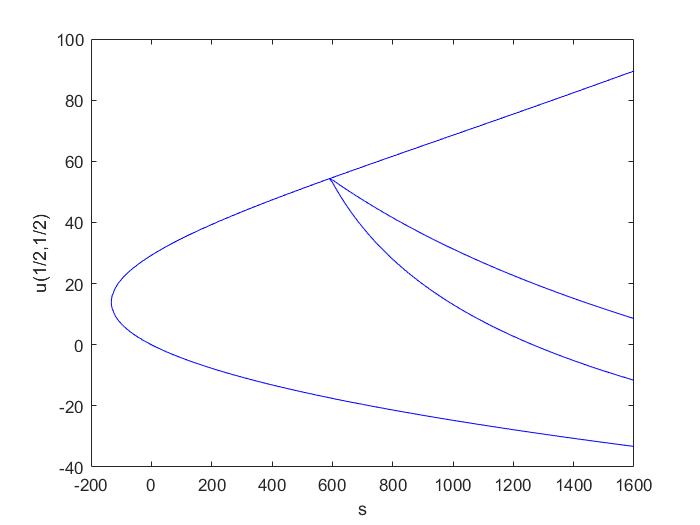}
   \caption{Bifurcation diagram of solutions of Eq. (\ref{2dex}) with respect to $s$.}
   \label{2dbif}
\end{figure}

\subsubsection{The Gray–Scott model in 2D}
The last example is the steady-state Gray-Scott model, given by
\begin{equation}
\begin{cases} D_A\Delta A=-SA^2+(\mu+\rho)A, \quad \text{on} \quad \Omega\\ D_S\Delta S=SA^2-\rho(1-S), \quad \text{on} \quad \Omega
\\ 
\displaystyle
\frac{\partial u}{\partial \textbf{n}}=\frac{\partial v}{\partial \textbf{n}}=0 \quad \text{on} \quad \partial \Omega

\end{cases}
\end{equation}
$\Omega=(0,1) \times (0,1)$ and $\textbf{n}$ is normal vector \cite{hao2020spatial}. The solutions of this model depend on the constants $D_A$, $D_S$, $\mu$, and $\rho$. In this example, we choose $D_A=2.5 \times 10^{-4}$, $D_S=5 \times 10^{-5}$, $\rho=0.04$, and $\mu=0.065$. Similar to the example of the Schnakenberg model in 1D (\ref{changed_from}), we can modify the system as follows:
\begin{equation} \label{2dgrayeq}
\begin{cases} D_A\Delta A+D_S\Delta S=(\mu+\rho)A-\rho(1-S), \\ D_S\Delta S=SA^2-\rho(1-S),\end{cases}
\end{equation}
As the first equation in the Gray-Scott model is linear, we can solve for $A$ in terms of $S$ and substitute into the second equation to obtain a single polynomial equation. 
When using the companion matrix on the coarsest grid ($\ell=0$) to solve the polynomial equation, we obtain $3^5$ complex solutions, leading to a large number of possible combinations on finer grids. To address this, we applied two approaches. The first approach involved keeping the real solutions and real parts of complex solutions, and using linear interpolation and Newton's method to refine them on finer grids. This approach yielded 24 solutions on a step size of $2^{-6}$, up to rotation shown in Fig. \ref{2dgray}. The second approach involved keeping only the real solutions, resulting in $1$ solution on $\ell=0$ and $3^{13}$ initial guesses on $\ell=1$. Using interpolation and Newton's method, we obtained 88 solutions on a step size of $2^{-6}$, up to rotation shown in Fig. \ref{2dgray1}. If we use more large coefficient for filtering conditions we can get 187 solutions.

\begin{figure}[!ht]
\centering
\includegraphics[width=10cm]{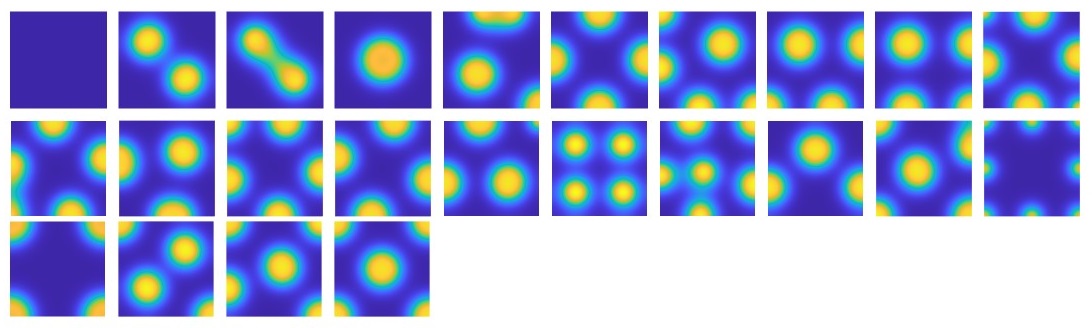}
\includegraphics[width=10cm]{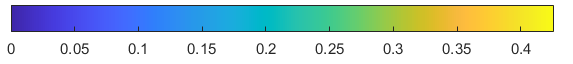}
  \caption{We have 24 solutions and plot only $A(x,y)$ from (\ref{2dgrayeq}) with a step size of $2^{-6}$. The initial guesses were refined by considering both real solutions and real parts of complex solutions on the coarsest grid ($\ell=0$).}
  \label{2dgray}
\end{figure}

\begin{figure}[!ht]
\centering
\includegraphics[width=10cm]{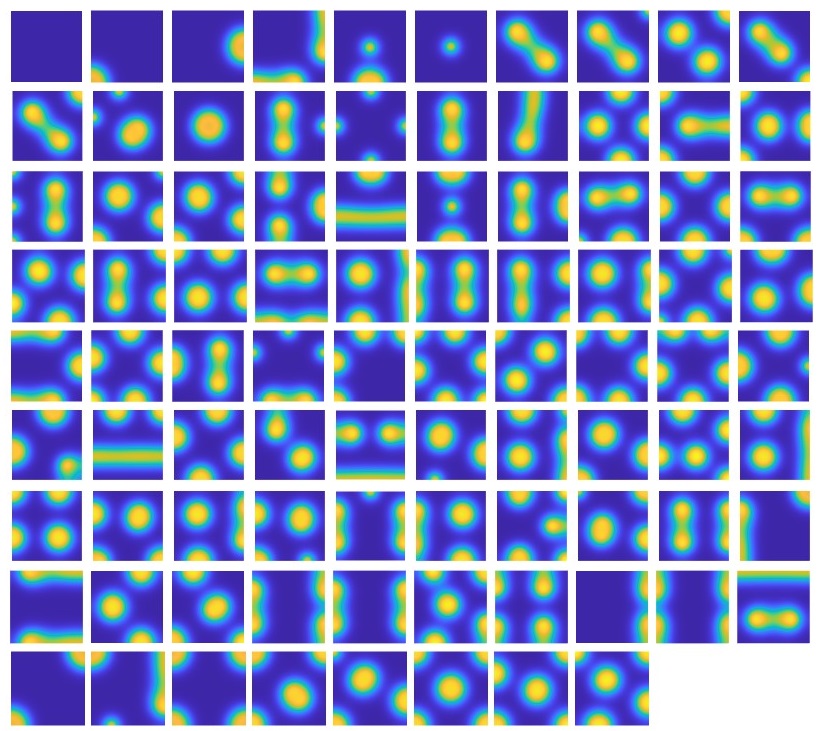}
\includegraphics[width=10cm]{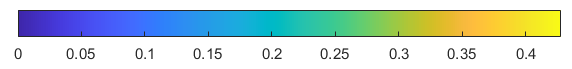}
  \caption{ We have 88 solutions and plot only $A(x,y)$ from (\ref{2dgrayeq}) with a step size of $2^{-6}$. The initial guesses were refined by considering only the real solutions on the $\ell=1$ grid.}
  \label{2dgray1}
\end{figure}

\section{Conclusion}\label{con} In this paper, we have presented a novel approach, the Companion-Based Multilevel finite element method (CBMFEM), which efficiently and accurately generates multiple initial guesses for solving nonlinear elliptic semi-linear equations with polynomial nonlinear terms. Our numerical results demonstrate the consistency of the method with theoretical analysis, and we have shown that CBMFEM outperforms existing methods for problems with multiple solutions.

Furthermore, CBMFEM has potential applications in more complex PDEs with polynomial nonlinear terms. To generalize our approach, we need to conduct further investigations to identify better filtering condition constants and better nonlinear solvers. In our future work, we shall incorporate the multigrid method to speed up the Newton method, which should further improve the efficiency of the method. Overall, CBMFEM is a promising approach for solving elliptic PDEs with multiple solutions, and we will apply it to widespread applications in various scientific and engineering fields.

{\bf Data availability} Data sharing is not applicable to this article as no datasets were generated or analyzed during
the current study.

{\bf Declarations}
 WH and SL is supported by NIH via 1R35GM146894. YL is supported by NSF via DMS 2208499 There is no conflict of interest.

\bibliographystyle{plain}
\bibliography{references}

%

\end{document}